\documentclass[12pt]{amsart}

\newtheorem{theorem}{Theorem}[section]

\newtheorem{proposition}[theorem]{Proposition}

\newtheorem{corollary}[theorem]{Corollary}

\usepackage[bookmarks]{hyperref}
\usepackage{cite}
\usepackage{graphicx,amscd,amsthm,amsfonts,amsopn,amssymb,verbatim,enumerate,xcolor}
\usepackage[letterpaper,left=1in,top=1in,right=1in,bottom=1in]{geometry}

\def\half{ \frac{1}{2}}
\def\D{\partial}

\def\R{{\mathbb R}}

\def\nint{\mathop{\diagup\kern-13.0pt\int}}

\def\bas{\begin{align*}}
\def\eas{\end{align*}}
\def\bi{\begin{itemize}}
\def\ei{\end{itemize}}

\def\emph#1{{\it #1}}

\def\FF{{\mathcal F}}

\def\eps{{\epsilon}}
\def\OO{{\mathcal O}}

\theoremstyle{definition}

\numberwithin{equation}{section}

\begin{document}

\title{Improved low regularity theory for gravity-capillary waves}
\author{Albert Ai}
\address{Department of Mathematics, University of Wisconsin, Madison}
\email{aai@math.wisc.edu}


\begin{abstract}
This article concerns the Cauchy problem for the gravity-capillary water waves system in general dimensions. We establish local well-posedness for initial data in $H^s$, with $s > \frac{d}{2} + 2 - \mu$, with $\mu = \frac{3}{14}$ and $\mu = \frac37$ in the cases $d = 1$ and $d \geq 2$ respectively. This represents an improvement over the the state-of-the-art low regularity theory in $d \geq 2$ dimensions.
\end{abstract}

\maketitle
\addtocontents{toc}{\protect\setcounter{tocdepth}{1}}
\tableofcontents

\section{Introduction}

We consider the Cauchy problem for the gravity-capillary water waves system in the context of an incompressible, irrotational fluid flow. 

Let $\Omega$ denote a time-dependent fluid domain contained in a fixed domain $\OO$, located between a free surface and a fixed bottom. Precisely,
\[
\Omega = \{(t, x, y) \in [0, 1] \times \OO \ ; \ y < \eta(t, x)\},
\]
where $\OO \subseteq \R^d \times \R$ is a given connected open set, $x \in \R^d$ represents the horizontal spatial coordinate, and $y \in \R$ represents the vertical spatial coordinate parallel to the force of gravity. 

We assume the free surface
$$\Sigma = \{(t, x, y) \in [0, 1] \times \R^d \times \R: y = \eta(t, x)\}$$
is separated from the fixed bottom $\Gamma = \D\Omega \backslash \Sigma$ by a curved strip of depth $h > 0$:
\begin{align}\label{width}
\{(x, y) \in \R^d \times \R: \eta(t, x) - h < y < \eta(t, x) \} \subseteq \OO.
\end{align}
This assumption is satisfied, for instance, in the case of infinite depth.

We consider an incompressible, irrotational fluid flow in the presence of gravity. In this setting the fluid velocity field $v$ may be given by $\nabla_{x, y} \phi$ where $\phi: \Omega \rightarrow \R$ is a harmonic velocity potential,
$$\Delta_{x, y} \phi = 0.$$
The water waves system is then given by Euler's equation in the presence of gravity, a kinematic condition requiring that fluid particles at the interface remain at the interface, the balance of forces at the interface, and a solid bottom constraint:
\begin{equation}\label{euler}
\begin{cases}
\begin{aligned}
&\D_t \phi + \dfrac{1}{2} |\nabla_{x, y}\phi|^2 + P + gy = 0 \qquad &\text{in } \Omega, \\
&\D_t \eta = \D_y \phi - \nabla_x \eta \cdot \nabla_x \phi \qquad &\text{on } \Sigma, \\
&P = \kappa H(\eta) \qquad &\text{on } \Sigma, \\
&\D_\nu \phi = 0 \qquad &\text{on } \Gamma.
\end{aligned}
\end{cases}
\end{equation}
Here, $g > 0$ is acceleration due to gravity, $\nu$ is the normal to $\Gamma$, $P$ is the pressure, $\kappa > 0$ is the coefficient of surface tension, and
\[
H(\eta) = \text{div}_x \left( \frac{\nabla_x \eta}{\sqrt{1 + |\nabla_x \eta|^2}} \right)
\]
is the mean curvature of the interface.

The free boundary problem (\ref{euler}) reduces to a system of equations on the free surface. Following Zakharov \cite{zak} and Craig-Sulem \cite{CSS}, we have unknowns $(\eta, \psi)$ where $\eta$ is the vertical position of the fluid surface as before, and
$$\psi(t, x) = \phi(t, x, \eta(t, x))$$
is the velocity potential $\phi$ restricted to the surface. Then (henceforth, $\nabla = \nabla_x$):
\begin{equation}\label{zak}
\begin{cases}
\D_t \eta - G(\eta) \psi = 0 ,\\
\displaystyle \D_t \psi + g\eta - \kappa H(\eta) + \half |\nabla \psi|^2 - \half \frac{(\nabla \eta \cdot \nabla \psi + G(\eta) \psi)^2}{1 + |\nabla \eta|^2} = 0,
\end{cases}
\end{equation}
where $G(\eta)$ is the Dirichlet to Neumann map with boundary $\eta$:
\[(G(\eta)\psi)(t, x) = \sqrt{1 + |\nabla \eta|^2} (\D_n \phi) |_{y = \eta(t, x)}.\]
In the sequel we will fix $g = \kappa = 1$ without losing generality.

\

For the well-posedness of \eqref{zak} in the context of smooth data, see Nalimov \cite{n}, Shinbrot \cite{Shinbrot}, Yoshihara \cite{y, y2}, Beyer-Gunther \cite{BG}, Coutand-Shkoller \cite{CS}, and Ming-Zhang \cite{MZ}.

In the context of well-posedness with limited Sobolev regularity, the scaling-critical space 
\[
(\eta, \psi) \in H^{s + \half} \times H^s, \qquad s = \frac{d}{2} + \frac32
\]
serves as a natural lower threshold above which we can hope to obtain well-posedness, although due to the quasilinear nature of the problem, the actual threshold is likely to be higher. The first result in this direction was established by Alazard-Burq-Zuily \cite{ABZcap}, who proved well-posedness in Sobolev spaces half a derivative above the scaling critical threshold, 
\[
s > \frac{d}{2} + 2.
\]
A key step for the proof of the energy estimates underlying this result was the use of a paradifferential reduction of \eqref{zak},
\begin{equation}\label{introparadiff}
\D_t u + T_V \cdot \nabla u + i T_\gamma u = f,
\end{equation}
where $\gamma$ is a real symbol of order $3/2$ We will review this reduction in greater detail in Section~\ref{lwp}.

To bring the well-posedness theory closer to the scaling threshold, we observe that \eqref{introparadiff} is dispersive of order 3/2, and in particular we expect solutions to satisfy Strichartz estimates. For the constant coefficient counterpart to \eqref{introparadiff},
\[
(\D_t + i|D_x|^{\frac32}) u = f,
\]
these estimates take the form
\begin{equation}\label{strichartz}
\|u\|_{L^p([0, T]; W^{s - \frac{d}{2} - \mu, \infty})} \lesssim \|u(0)\|_{H^s} + \|f\|_{L^1([0, T]; H^s)}
\end{equation}
with
\begin{equation}\label{strich-numerology}
\begin{aligned}
p = 4, \quad \mu &= \frac{3}{8}, \qquad \text{if} \quad d = 1,\\
p = 2, \quad \mu &< \frac{3}{4}, \qquad \text{if} \quad d \geq 2.
\end{aligned}
\end{equation}

It was first observed in the context of the nonlinear wave equation that using Strichartz estimates in the place of Sobolev embedding within the energy method reduces the amount of Sobolev regularity required for local well-posedness; see the works of Bahouri-Chemin \cite{BC0, bc}, Tataru \cite{T, T2, T3}, Klainerman-Rodnianski \cite{KR}, and Smith-Tataru \cite{ST}. 

More recently, Strichartz estimates have also been applied to the low regularity theory in the context of the water waves system. In the pure gravity water waves setting, where the surface tension coefficient $\kappa = 0$, see the works of Alazard-Burq-Zuily \cite{abz-str} and of the author \cite{A, A1}. 

\

In the current article, we are concerned with the application of Strichartz estimates to the low regularity well-posedness of the gravity-capillary system \eqref{zak} with nonvanishing $\kappa$. Similar to the previous programs of research involving the nonlinear wave and pure gravity equations, the key difficulty with this approach is proving Strichartz estimates for equations with limited regularity coefficients; in this case, \eqref{introparadiff}. Depending on the regularity of these coefficients, the Strichartz estimate \eqref{strichartz} may only hold on microlocal time scales of length less than 1. This results in a loss relative to the constant coefficient gain \eqref{strich-numerology}.

In the setting of 2D water waves with $d = 1$, Christianson-Hur-Staffilani \cite{CHS} proved Strichartz estimates for smooth solutions (requiring $s \geq 15$), while Alazard-Burq-Zuily \cite{ABZcapS} established the estimates for solutions at the regularity produced by the energy method, half a derivative above scaling. Both of these results incurred a 1/8 derivative loss relative to \eqref{strich-numerology}. Precisely, they showed that solutions
\[
(\eta, \psi) \in C^0([0, T]; H^{s + \half}(\R) \times H^s (\R)), \qquad s > \frac{d}{2} + 2,
\]
to \eqref{zak} satisfy
\[
(\eta, \psi) \in L^4([0, T]; B^{s + \half - \mu}_{\infty, 2}(\R) \times B^{s - \mu}_{\infty, 2}(\R)), \qquad \mu = \frac14,
\]
where $B^\sigma_{\infty, 2}$ denotes the usual Besov space.

In this result, the authors took advantage of the one dimensional setting to apply a paracompositional change of variables, flattening the coefficient of the order $3/2$ dispersive term. Nguyen \cite{Nsharp} later refined this paracompositional approach to establish a refined estimate using H\"older control norms, and thus was able to establish well-posedness for data with the same Strichartz gain,
\[
s > \frac{d}{2} + 2 - \mu, \qquad \mu = \frac14.
\]

\

In the general dimensional setting, one does not have access to the para-compositional change of variables. As a result, to extend the low regularity theory below the energy threshold, the Strichartz estimates need to be proven for equations with even lower coefficient regularity. This was successfully performed in the general dimensional setting by de Poyferre-Nguyen \cite{PN0, PN}, by conducting the analysis at microlocal scales even narrower than that of \cite{ABZcapS}. However, this induces a further loss in the Strichartz estimates, and accordingly yields a narrower gain in the low regularity theory. Precisely, they proved an estimate of the form
\[
\|u\|_{L^p([0, T]; W^{s - \frac{d}{2} - \mu, \infty})} \lesssim \|u_0\|_{H^s} + \|f\|_{L^p([0, T]; H^s)}
\]
with
\begin{equation}\label{pn-result}
\begin{aligned}
p = 4, \quad \mu &< \frac{3}{20}, \qquad \text{if} \quad d = 1,\\
p > 2, \quad \mu &< \frac{3}{10}, \qquad \text{if} \quad d \geq 2,
\end{aligned}
\end{equation}
with the corresponding well-posedness at regularities
\[
s > \frac{d}{2} + 2 - \mu.
\]
This provides a weaker result than the approach using the para-compositional change of variables, but is able to extend to all dimensions $d \geq 2$.

\

In the current article, we improve the well-posedness result in general dimensions $d \geq 2$: 

\begin{theorem}\label{t:main}
Let $d \geq 1$,
\begin{equation*}
\begin{aligned}
p = 4, \quad \mu &< \frac{3}{14}, \qquad \text{if} \quad d = 1,\\
p > 2, \quad \mu &< \frac{3}{7}, \qquad \text{if} \quad d \geq 2,
\end{aligned}
\end{equation*}
and 
\[
s > \frac{d}{2} + 2 - \mu, \qquad 2 < r < s - \frac{d}{2} + \mu.
\]
Consider initial data $(\eta_0, \psi_0) \in H^{s + \half}(\R^d) \times H^s(\R^d)$ satisfying the positive depth condition
\[
\{(x, y) \in \R^d \times \R: \eta_0(x) - h < y < \eta_0(x) \} \subseteq \OO.
\]

Then there exists $T > 0$ such that the system \eqref{zak} with initial data $(\eta_0, \psi_0)$ has a unique solution $(\eta, \psi) \in C([0, T]; H^{s + \half} \times H^s)$, which satisfies
\begin{enumerate}[i)]
\item the positive depth condition \eqref{width} with $h/2$ in place of $h$,

\item the Strichartz estimate 

\[
(\eta, \psi) \in L^p([0, T]; W^{r + \half, \infty} \times W^{r, \infty}),
\]

\item and continuity of the solution map 
\[
(\eta_0, \psi_0) \mapsto (\eta, \psi), \qquad H^{s + \half} \times H^s \rightarrow C([0, T]; H^{s + \half} \times H^s).
\]

\end{enumerate}
\end{theorem}

We observe that the well-posedness result of Theorem~\ref{t:main} still exhibits a loss of derivatives relative to the Strichartz estimates \eqref{strich-numerology} obtained in the constant coefficient setting. This reflects the fact that we still reduce the dispersive analysis to microlocal time scales by partitioning the unit time interval.

However, in contrast to previous works in the capillary-gravity setting, which partition the time interval into equal-length intervals, we refine the partitioning procedure to account for the varying size of the background metric. This approach was first applied in the context of the wave equation by Tataru \cite{T3}, and in the context of pure gravity water waves by the author \cite{A}.

\

We outline the paper as follows. In Section~\ref{s:parametrix}, we review a wave packet parametrix construction via a phase space transform. We apply this construction in Section~\ref{s:strich} to establish dispersive and Strichartz estimates on unit time intervals for evolution equations under general symbol classes with sufficient regularity. In Section~\ref{s:gen-strich}, we generalize these Strichartz estimates to settings with lower regularity symbols, but with a loss of derivatives. Finally, in Section~\ref{lwp}, we return to the capillary-gravity system, and recall the paradifferential reduction of Alazard-Burq-Zuily and de Poyferre-Nguyen. We show that this paradifferential setting is amenable to the low regularity in which the Strichartz estimates of Section~\ref{s:gen-strich} apply, concluding the proof of Theorem~\ref{t:main}.

\subsection{Acknowledgements} 

The author was supported by the National Science Foundation under Grant No. DMS-2220519 and the RTG in Analysis and Partial Differential equations grant DMS-2037851. The author was also supported by the National Science Foundation under Grant No. DMS-1928930 while in residence at the Simons Laufer Mathematical Sciences Institute (formerly MSRI) in Berkeley, California, during the summer of 2023.

The author would like to thank Mihaela Ifrim for many helpful discussions.

\section{The parametrix construction}\label{s:parametrix}

In this section we consider a general evolution equation of the form
\begin{equation}\label{general}
\begin{cases}
(i\D_t + a^w(t, x, D))u = 0, \qquad &\text{in } (0, T) \times \R^d, \\
u(0) = u_0, \qquad &\text{on } \R^d,
\end{cases}
\end{equation}
where $a(t, x, \xi)$ is a real symbol continuous in $t$ and smooth with respect to $x$ and $\xi$. Here $a^w$ denotes a Weyl-quantized operator. In particular, it is self-adjoint and thus generates an isometric evolution $S(t, s)$ on $L^2(\R^d)$. 

We outline the construction of a phase space representation of the fundamental solution for (\ref{general}) using the FBI transform, following \cite{tataru2004phase}. For a more thorough discussion of the properties of the FBI transform, the reader may refer to the comprehensive exposition of Delort \cite{FBId}.

The FBI transform,
\[(Tf)(x, \xi) = 2^{-\frac{d}{2}} \pi^{-\frac{3d}{4}} \int e^{-\half(x - y)^2}e^{i\xi(x - y)} f(y)\, dy,\]
is an isometry from $L^2(\R^d)$ to the phase space $L^2(\R^{2d})$ with an inversion formula
\[f(y) = (T^*Tf)(y) = 2^{-\frac{d}{2}} \pi^{-\frac{3d}{4}} \int e^{-\half(x - y)^2}e^{-i\xi(x - y)} (Tf)(x, \xi) \, dx d\xi.\]
We can use the FBI transform to quantify the phase space localization of $S(t, s)$ around the Hamilton flow corresponding to \eqref{general},
\begin{equation}\begin{cases}\label{hamilton}
\dot{x} = a_\xi(t, x, \xi), \\
\dot{\xi} = -a_x(t, x, \xi).
\end{cases}\end{equation}

More precisely, let 
$$(x^t, \xi^t) = (x^t(x, \xi), \xi^t(x, \xi))$$
denote the solution to (\ref{hamilton}), with initial data $(x, \xi)$ at time $t = 0$. In addition, let $\chi(t, s)$ denote the family of transformations on the phase space $L^2(\R^{2d})$ corresponding to (\ref{hamilton}),
$$\chi(t, s)(x^s, \xi^s) = (x^t, \xi^t).$$
It was shown in \cite{tataru2004phase} that for the class of symbols $a \in S_{0, 0}^{0, (k)}$ defined by
\begin{equation}\label{S2}
|\D_x^\alpha \D_\xi^\beta a(t, x, \xi)| \leq c_{\alpha ,\beta}, \qquad |\alpha| + |\beta| \geq k,
\end{equation}
the flow satisfies the following properties:
\begin{theorem}\label{t:bilipthm}
Let $a(t, x, \xi) \in S_{0, 0}^{0, (2)}$. Then
\begin{enumerate}
\item The Hamilton flow \eqref{hamilton} is well-defined and bilipschitz.

\item The kernel $\tilde{K}(t,s)$ of the phase space operator $TS(t, s)T^*$ decays rapidly away from the graph of the Hamilton flow,
\[
|\tilde{K}(t, x, \xi, s, y, \eta)| \lesssim (1 + |(x, \xi) - \chi(t, s)(y, \eta)|)^{-N}.
\]
\end{enumerate}
\end{theorem}
Then we have the following phase space representation for solutions to (\ref{general}), as a consequence of \cite[Theorem 4]{tataru2004phase}:
\begin{theorem}\label{t:repformula}
Let $a(t, x, \xi) \in S_{0, 0}^{0, (2)}$. Then the kernel $K(t, s)$ of the evolution operator $S(t, s)$ for $i\D_t + a^w$ can be represented in the form 
$$K(t, y, s, \tilde{y}) = \int e^{-\half(\tilde{y} - x^s)^2} e^{-i\xi^s(\tilde{y} - x^s)} e^{i(\psi(t, x, \xi) - \psi(s, x, \xi))} e^{i\xi^t(y - x^t)} G(t, s, x, \xi, y) \, dx d\xi ,$$
where the function $G$ satisfies
$$|(x^t - y)^\gamma \D_x^\alpha \D_\xi^\beta \D_y^\nu G(t, s, x, \xi, y)| \lesssim c_{\gamma, \alpha, \beta, \nu}.$$
\end{theorem}

Theorems \ref{t:bilipthm} and \ref{t:repformula} were generalized in \cite{marzuola2008wave} to the class of symbols $a\in S^{(k)}L_\chi^1$ with $k = 2$, satisfying
$$\sup_{x, \xi} \int_0^1|\D_x^\alpha \D_\xi^\beta a(t, \chi(t, 0)(x, \xi))| \, dt \leq c_{\alpha, \beta}, \qquad |\alpha| + |\beta| \geq k.$$
For our purposes it suffices to consider an intermediate class of symbols $a\in L^1S_{0, 0}^{0, (2)}$ satisfying
$$\|\D_x^\alpha \D_\xi^\beta a\|_{L^1_t([0, T]; L^\infty)} \leq c_{\alpha, \beta}, \qquad |\alpha| + |\beta| \geq 2.$$

\section{Strichartz estimates}\label{s:strich}

In this section, we combine the representation formula in Theorem \ref{t:repformula} with a curvature condition to yield a dispersive estimate, from which Strichartz estimates follow via a classical $TT^*$ method.

First, we define a class of symbols in analogy with the class $\lambda S_\lambda^k$ in \cite{koch2005dispersive}. For $m \in \R$, $k > 0$, and $\delta \in [0, 1]$, the class of symbols $a(t, x, \xi) \in L^1S_{1, \delta}^{m, (k)}(\lambda)$ satisfies
\begin{equation}\label{L1-symbol}
\begin{aligned}
&\|\D_x^\alpha \D_\xi^\beta a\|_{L^1_t([0, 1]; L^\infty)} \leq c_{\alpha, \beta} \lambda^{m - |\beta| + \delta(|\alpha| - k)}, \qquad &|\alpha| \geq k.
\end{aligned}
\end{equation}
We also require the symbols to partially satisfy bounds uniform in time. Define the class of symbols $a \in S_1^m(\lambda)$ satisfying
\begin{equation}\label{Linfty-symbol}
\|\D_\xi^\beta a\|_{L^\infty_t([0, 1]; L^\infty)} \leq c_\beta \lambda^{m - |\beta|}.
\end{equation}

\begin{proposition}\label{p:dispersive}
Let $m \in [1, 2]$, $\delta = \frac{2 - m}{2}$, and consider the symbol
\begin{equation}\label{symbol-class}
a \in S_1^m(\lambda), \qquad \lambda^{2m - 2} a \in L^1S_{1, \delta}^{m, (2)}(\lambda),
\end{equation}
satisfying also
\begin{equation}\label{symbol-convexity}
| \det \D_\xi^2a (t, x, \xi)| \geq c\lambda^{d(m - 2)} \qquad \text{on} \qquad  [0, 1] \times \R^d \times \{|\xi| \approx \lambda\}.
\end{equation}

Let $u_0$ have frequency support $\{|\xi| \approx \lambda\}$. Then for all $|t - s| \leq 1$, we have
\[
\|S(t, s) u_0\|_{L^\infty} \lesssim \lambda^{\delta d} |t - s|^{-\frac{d}{2}} \|u_0\|_{L^1}.\]
\end{proposition}

\begin{proof}
Without loss of generality let $s = 0$. We fix $\tau \in [0, 1]$ and prove the estimate when $t = \tau$. To do so, we reduce the problem to an estimate for $t = 1$ by rescaling. Write $u = S(t, 0)u_0$ and set
\[
v(t, x) = u(\tau t, \mu x), \qquad \mu = \tau^{\frac12} \lambda^{-\delta}.
\]
Then $v$ solves
$$(i\D_t + \tilde{a}(t, x, D)) v = 0, \qquad v(0) = u_0(0, \mu x) =: v_0$$
where
$$\tilde{a}(t, x, \xi) = \tau a(\tau t, \mu x, \mu^{-1} \xi),$$
and it suffices to show
$$\|v(1)\|_{L^\infty} \lesssim \|v_0\|_{L^1}.$$

We first show $\tilde{a} \in L^1S_{0, 0}^{0, (2)}$ in order to apply Theorem \ref{t:repformula}. First observe that we may assume $\tau^{-1} \lambda^{-m} \leq 1$, otherwise the dispersive estimate is immediate by Sobolev embedding. When $|\alpha| \geq 2$, we have, since $\lambda^{2m-2}a \in L^1S_{1, \delta}^{m, (2)}(\lambda)$,
\begin{equation}
\begin{aligned} \label{tildeabd}
\|\D_x^\alpha\D_\xi^\beta \tilde{a}\|_{L^1_t([0, 1]; L^\infty)} &= \mu^{|\alpha| - |\beta|} \|\D_x^\alpha\D_\xi^\beta a\|_{L^1_t([0, \tau]; L^\infty)}  \\
&\lesssim \mu^{|\alpha| - |\beta|} \cdot \lambda^{2 - 2m} \cdot \lambda^{m - |\beta| + \delta(|\alpha| - 2)}\\
& = \tau^{\frac12 |\alpha|} (\tau^{-1} \lambda^{-m})^{\half |\beta|} \lesssim 1.
\end{aligned}
\end{equation}
If $|\alpha| = 0$, we have $|\beta| \geq 2$ and we instead use $a \in S_1^m(\lambda)$ to estimate
\begin{equation}\label{tildeabd3}
\begin{aligned}
\|\D_\xi^\beta \tilde{a}\|_{L^1_t([0, 1]; L^\infty)} &= \mu^{- |\beta|} \|\D_\xi^\beta a\|_{L^1_t([0, \tau]; L^\infty)} \\
&\lesssim \mu^{- |\beta|} \tau\|\D_\xi^\beta a\|_{L^\infty_t([0, \tau]; L^\infty)} \\
&\lesssim \mu^{- |\beta|} \tau \lambda^{m - |\beta|} \\
&= (\tau^{-1}\lambda^{-m})^{\frac12 (|\beta| - 2)} \lesssim 1.
\end{aligned}
\end{equation}
Lastly, if $|\alpha| = 1$, we have by interpolation,
\begin{equation}
\begin{aligned} \label{tildeabd2}
\|\D_x^\alpha\D_\xi^\beta \tilde{a}\|_{L^1_t([0, 1]; L^\infty)} &= \mu^{1 - |\beta|} \|\D_x^\alpha\D_\xi^\beta a\|_{L^1_t([0, \tau]; L^\infty)}  \\
&\lesssim \mu^{1 - |\beta|} \cdot \lambda^{1 - m} \cdot \lambda^{m - |\beta|}\\
& = (\tau^{-1} \lambda^{-m})^{\half (|\beta| - 1)} \lesssim 1.
\end{aligned}
\end{equation}
This completes the proof that $\tilde{a} \in L^1S_{0, 0}^{0, (2)}$.

\

We may then apply the representation formula in Theorem \ref{t:repformula},
\begin{equation*}
\begin{aligned}
v(t, y)  &= \int e^{-\half(\tilde{y} - x)^2} e^{-i\xi(\tilde{y} - x)} e^{i\psi(t, x, \xi)} e^{i\xi^t(y - x^t)} G(t, 0, x, \xi, y)  v_0(\tilde{y}) \, dx d\xi d\tilde{y}.
\end{aligned}
\end{equation*}
By the frequency support of $v_0$ in $B := \{|\xi| \approx \mu\lambda \}$, the contribution of the complement of $B$ to the integral is negligible, so it suffices to consider
$$\int\int_{B} |G(t, x, \xi, y)| \, d\xi \, e^{-\half(\tilde{y} - x)^2} |v_0(\tilde{y})| \, dx d\tilde{y} \lesssim \|v_0\|_{L^1}\sup_x \int_B |G(t, x, \xi, y)| \, d\xi.$$
It remains to show
$$\int_B |G(1, x, \xi, y)| \, d\xi \lesssim 1.$$
Given the bound for $G$ in Theorem \ref{t:repformula}, this reduces to showing
$$\int_B (1 + |x^1 - y|)^{-N} \, d\xi \lesssim 1.$$

We consider the dependence of $x^1 = x^1(x, \xi)$ on $\xi$. Denote
$$(X, \Xi) := (\D_\xi x^t, \D_\xi \xi^t),$$
which by the Hamilton flow (\ref{hamilton}) for $\tilde{a}$ solves
\begin{equation}
\begin{cases}
\dot{X} = \tilde{a}_{\xi x} X + \tilde{a}_{\xi \xi} \Xi, \qquad X(0) = 0 ,\\
\dot{\Xi} = -\tilde{a}_{xx} X - \tilde{a}_{x \xi} \Xi, \qquad \Xi(0) = I.
\end{cases}
\end{equation}
From \eqref{tildeabd}, \eqref{tildeabd3}, and \eqref{tildeabd2}, we have
\[\|\tilde{a}_{x\xi}\|_{L^1_t([0, 1]; L^\infty)} + \|\tilde{a}_{\xi\xi}\|_{L^1_t([0, 1]; L^\infty)} + \|\tilde{a}_{xx}\|_{L^1_t([0, 1]; L^\infty)} \lesssim 1.\]
We conclude
$$\|X\|_{L^\infty} + \|\Xi\|_{L^\infty} \lesssim 1.$$
In turn, we have
$$\Xi(t) = \Xi(0) + \int_0^t \dot{\Xi} \, ds = I - \int_0^t \tilde{a}_{xx} X + \tilde{a}_{x \xi} \Xi \, ds = I + O(1),$$
$$\tilde{a}_{\xi \xi} \Xi = \tilde{a}_{\xi \xi} + O(1),$$
and hence
\begin{equation}\label{Xflow}
X(t) = X(0) + \int_0^t \dot{X} \, ds = \int_0^t \tilde{a}_{\xi x} X + \tilde{a}_{\xi \xi} \Xi \, ds = \int_0^t \tilde{a}_{\xi \xi} \, ds + O(1).
\end{equation}
We conclude
$$|\det \D_\xi x^1| = |\det X(1)| = \left |\int_0^1 \det \tilde{a}_{\xi \xi} \, ds \right| + O(1) \gtrsim 1,$$
and hence we may change variables to obtain
$$\int_B (1 + |x^1 - y|)^{-N} \, d\xi \lesssim \int (1 + |x^1 - y|)^{-N} \, dx^1 \lesssim 1$$
as desired.
\end{proof}

\subsection{Strichartz estimates}

A Strichartz estimate follows from Proposition~\ref{p:dispersive} and a standard $TT^*$ argument (see for instance \cite{KTStrich}):

\begin{theorem}\label{t:strichartz}
Let $m \in [1, 2]$, $\delta = \frac{2 - m}{2}$, and symbol $a(t, x, \xi)$ satisfy \eqref{symbol-class} and \eqref{symbol-convexity}. Let $u$ have frequency support $\{|\xi| \approx \lambda\}$ and solve
\begin{equation}\label{inhomog}
\begin{cases}
(i\D_t + a^w(t, x, D))u = f, \qquad &\text{in } (0, 1) \times \R^d, \\
u(0) = u_0, \qquad &\text{on } \R^d.
\end{cases}
\end{equation}
Then for $p, q$ satisfying
\begin{equation}\label{pq}
\frac{2}{p} + \frac{d}{q} = \frac{d}{2}, \quad 2 \leq p \leq \infty, \quad 1 \leq q \leq \infty, \quad (p, q) \neq (2, \infty),
\end{equation}
we have
\begin{equation}\label{strich-estimate}
\|u\|_{L^p([0, 1];L^q)} \lesssim \lambda^{\frac{2\delta}{p}} (\|u\|_{L^\infty([0, 1];L^2)} + \|f\|_{L^1([0, 1];L^2)}).
\end{equation}
\end{theorem}

\section{Strichartz estimates for rough coefficients}\label{s:gen-strich}

In this section we generalize the Strichartz estimate of Theorem~\ref{t:strichartz}, considering dependence on lifespan $T$ and introducing a parameter $\mu$ which we will subsequently use to track variable symbol regularity.

\begin{theorem}\label{t:var-strich}
Let $m \in [1, 2]$. Consider a symbol $a(t, x, \xi)$ satisfying 
\[
\|\D_\xi^\beta a\|_{L_t^\infty([0, T]; L^\infty_{x, \xi})} \lesssim \lambda^{m - |\beta|}, \qquad T\|\D_\xi^\beta a\|_{L_t^1([0, T]; L^\infty_\xi \dot C_x^{2})} \lesssim (\mu \lambda^{1 - m})^2 \lambda^{m - |\beta|},
\]
and \eqref{symbol-convexity} on $t \in [0 ,T]$. Let $u$ have frequency support $\{|\xi| \approx \lambda\}$ and solve \eqref{inhomog} on $t \in [0, T]$. Then for $p, q$ satisfying \eqref{pq}, we have
\begin{equation}\label{var-strich}
\|u\|_{L^p([0, T];L^q)} \lesssim \lambda^{\frac{2\delta}{p}}\mu^{\frac{1}{p}} (\|u\|_{L^1([0, T];L^2)} + \mu^{-1}\|f\|_{L^1([0, T];L^2)}).
\end{equation}
\end{theorem}

\begin{proof}

We establish the general estimate via a series of generalizations.

\

\emph{Step 1.} We establish the case $\mu = 1$ and $T = 1$. Writing $I = [0, 1]$, we have the truncation error
\begin{equation*}
\begin{aligned}
\|a_{> \lambda^\delta }u\|_{L^1(I;L^2)} &\lesssim \lambda^{m} \cdot \|\lambda^{-m} a_{> \lambda^\delta}\|_{L^1(I; L^\infty)} \|u\|_{L^\infty(I; L^2)} \\
&\lesssim \lambda^m \lambda^{-2\delta} \cdot \|\lambda^{-m} a\|_{L_t^1(I; L^\infty_\xi \dot C_x^{2})} \|u\|_{L^\infty(I; L^2)} \\
&\lesssim \|u\|_{L^\infty(I; L^2)}.
\end{aligned}
\end{equation*}
Absorbing this into $f$, it is easy to check that $\lambda^{2m - 2} a_{\leq \lambda^{\delta}} \in L^1S_{1, \delta}^{m, (2)}(\lambda)$, so we may apply Theorem~\ref{t:strichartz}.

\

\emph{Step 2.} We next establish the case $\mu = 1$ and general lifespan $T$. We apply the scaling 
\[
\tilde a = T a(Tt, T^\frac{1}{m} x, T^{-\frac{1}{m}} \xi), \qquad \tilde u = u(Tt, T^\frac{1}{m} x), \qquad \tilde f = Tf(Tt, T^\frac{1}{m} x)
\]
with $\tilde \lambda = T^{\frac{1}{m}} \lambda$. Then
\[
\|\D_\xi^\beta \tilde a\|_{L_t^1([0, 1]; L^\infty_\xi \dot C_x^{2})} = T^{\frac{2}{m} - \frac{|\beta|}{m}} \|\D_\xi^\beta a\|_{L_t^1([0, T]; L^\infty_\xi \dot C_x^{2})} \lesssim T^{\frac{2}{m} - \frac{|\beta|}{m}} T^{-1} \lambda^{2 - 2m} \lambda^{m - |\beta|} = \tilde \lambda^{2 - 2m} \tilde\lambda^{m - |\beta|},
\]
and 
\[
\|\D_\xi^\beta \tilde a\|_{L_t^\infty([0, 1]; L^\infty_{x, \xi})} \lesssim T^{1 - \frac{|\beta|}{m}}\|\D_\xi^\beta a\|_{L_t^\infty([0, T]; L^\infty_{x, \xi})} \lesssim \tilde \lambda^{m - |\beta|},
\]
reducing to the $\mu = 1$, $T = 1$ case.

\

\emph{Step 3.} We establish the case with general $\mu$ and $T = 1$. Decompose $[0, 1]$ into maximal subintervals 
\[
0 = t_0 < t_1 < ... < t_k = 1
\]
satisfying both
\begin{equation}\label{inhomog-split}
\|f\|_{L^1([t_j, t_{j + 1}]; L^2)} \leq \mu^{-1}\|f\|_{L^1([0, 1]; L^2)}
\end{equation}
and
\begin{equation}\label{symbol-split}
(t_{j + 1} - t_j) \cdot \lambda^{-2(1 - m)} \lambda^{|\beta| - m} \|\D_\xi^\beta a\|_{L^1([t_j, t_{j + 1}];L^\infty_\xi C_x^2)} \leq 1
\end{equation}
for each $0 \leq |\beta| \leq N$. By \eqref{inhomog-split}, the number $k$ of intervals satisfies
\[
k \geq \mu.
\]

In fact, we show that $k \approx \mu$. For the upper bound on $k$, observe that by maximality, each interval must satisfy equality in at least one of \eqref{inhomog-split} or \eqref{symbol-split} for some $\beta$. The number $k^*$ of intervals with equality in \eqref{inhomog-split} is at most $\mu$. On the other hand, for an interval with equality in \eqref{symbol-split}, since $c^{-1} + c \gtrsim 1$ for any $c$,
\[
\mu (t_{j + 1} - t_j) + \mu^{-1} \lambda^{-2(1 - m)} \lambda^{|\beta| - m} \|\D_\xi^\beta a\|_{L^1([t_j, t_{j + 1}];L^\infty_\xi C_x^2)} \gtrsim 1.
\]
Summing over all $k_\beta$ such intervals, we have
\[
\mu + \mu^{-1} \mu^2 \gtrsim k_\beta,
\]
and thus
\[
k = k^* + \sum_\beta k_\beta \lesssim \mu + \mu^{-1} \mu^2 \lesssim \mu.
\]

On each interval $[t_j, t_{j + 1}]$, we are in the setting of Step 2 and thus have
\[
\|u\|_{L^p([t_j, t_{j + 1}];L^q)} \lesssim \lambda^{\frac{2\delta}{p}} (\|u_0\|_{L^2} + \|f\|_{L^1([t_j, t_{j + 1}];L^2)}).
\]
Using \eqref{inhomog-split} we have
\[
\|u\|_{L^p([t_j, t_{j + 1}];L^q)} \lesssim \lambda^{\frac{2\delta}{p}} (\|u_0\|_{L^2} + \mu^{-1} \|f\|_{L^1([0, 1]; L^2)}).
\]
Raising to the power $p$ and summing over the $\mu$ intervals, we obtain the desired estimate.

\

\emph{Step 4.} For the case of general $\mu$ and lifespan $T$, we scale as in Step 2.
\end{proof}

As an immediate consequence, we have a Strichartz estimate with loss for symbols with lower regularity:

\begin{corollary}\label{c:var-strich-rough}
Let $m \in [1, 2]$. Consider a symbol $a(t, x, \xi)$ satisfying for $r \in [0, 2]$,
\[
\|\D_\xi^\beta a\|_{L_t^\infty([0, 1]; L^\infty_{x, \xi})} + \|\D_\xi^\beta a\|_{L_t^1([0, 1]; L^\infty_\xi \dot C_x^{r})} \lesssim \lambda^{m - |\beta|},
\]
and \eqref{symbol-convexity} on $t \in [0 ,1]$. Let $u$ have frequency support $\{|\xi| \approx \lambda\}$ and solve \eqref{inhomog} on $t \in [0, 1]$. Then for $p, q$ satisfying \eqref{pq} and
\[
\mu = \lambda^{m - 1}\lambda^{\frac{2 - r}{2 + r}},
\]
we have
\begin{equation}\label{var-strich-rough}
\|u\|_{L^p([0, 1];L^q)} \lesssim \lambda^{\frac{2\delta}{p}}\mu^{\frac{1}{p}} (\|u\|_{L^1([0, 1];L^2)} + \mu^{-1}\|f\|_{L^1([0, 1];L^2)}).
\end{equation}
\end{corollary}

\begin{proof}

Writing $I = [0, 1]$ and truncating at frequency $\lambda^\sigma$, we have
\[
\|\D_\xi^\beta a_{\lambda^\sigma} \|_{L_t^1(I; L^\infty_\xi \dot C_x^{2})} \lesssim \lambda^{\sigma (2 - r)}\|\D_\xi^\beta a_{\lambda^\sigma} \|_{L_t^1(I; L^\infty_\xi \dot C_x^{r})} \lesssim \lambda^{\sigma (2 - r)} \lambda^{m - |\beta|},
\]
so we are in the setting of Theorem~\ref{t:var-strich} with 
\begin{equation}\label{mueqn}
(\mu \lambda^{1 - m})^2 = \lambda^{\sigma (2 - r)}
\end{equation}
and the inhomogeneous truncation error
\begin{equation*}
\begin{aligned}
\|a_{> \lambda^\sigma }u\|_{L^1(I;L^2)} &\lesssim \lambda^{m} \cdot \|\lambda^{-m} a_{> \lambda^\sigma}\|_{L^1(I; L^\infty)} \|u\|_{L^\infty(I; L^2)} \\
&\lesssim \lambda^m \lambda^{-r\sigma} \cdot \|\lambda^{-m} a\|_{L_t^1(I; L^\infty_\xi \dot C_x^{r})} \|u\|_{L^\infty(I; L^2)} \\
&\lesssim \lambda^{m - r\sigma} \|u\|_{L^\infty(I; L^2)}.
\end{aligned}
\end{equation*}
This may be absorbed into the first term on the right hand side of \eqref{var-strich} provided that
\[
\mu^{-1}\lambda^{m - r\sigma} \lesssim 1.
\]
Combined with \eqref{mueqn}, this takes the form
\[
\lambda^{m - r\sigma} \lesssim \mu = \lambda^{m - 1}\lambda^{\half \sigma (2 - r)},
\]
which holds if $\sigma = \frac{2}{2 + r}$.

\end{proof}

\section{Well-posedness for capillary-gravity water waves}\label{lwp}

We recall the paradifferential reduction 
\begin{equation}\label{paradiff}
\D_t u + T_V \cdot \nabla u + i T_\gamma u = f
\end{equation}
of the gravity-capillary system \eqref{zak} established by Alazard-Burq-Zuily \cite{ABZcap} and later refined by de Poyferre-Nguyen\cite{PN0}. Here,
\[
\gamma = \sqrt{\ell \lambda},
\]
which symmetrizes symbols of the Dirichlet-to-Neumann map and the mean-curvature operator,
\[
\lambda = \sqrt{(1 + |\nabla \eta|^2) |\xi|^2 - (\nabla \eta \cdot \xi)^2}, \qquad \ell = (1 + |\nabla \eta|^2)^{-\half} \left(|\xi|^2 - \frac{(\nabla \eta \cdot \xi)^2}{1 + |\nabla \eta|^2} \right),
\]
respectively. For the symbols of the unknown,
\[
u = T_p \eta + i T_q (\psi - T_B \eta),
\]
where
\[
q = (1 + |\nabla \eta|^2)^{-\half}, \quad p = (1 + |\nabla \eta|^2)^{-\frac54}|\xi|^{\half} + F(\nabla \eta, \xi)\D_x^\alpha \eta,
\]
and $F$ is smooth away from the origin, of order $-1/2$ in $\xi$ and $|\alpha| = 2$. Lastly, we have used $(V, B)(t, x)$ to denote the horizontal and vertical components of the velocity field restricted to the surface,
\[
V = (\nabla \phi)|_{y = \eta(t, x)}, \quad B = (\D_y \phi) |_{y = \eta(t, x)}.
\]

Refining the result of Alazard-Burq-Zuily \cite{ABZcap}, de Poyferre-Nguyen established the following estimate on the errors arising from the reduction of \eqref{zak} to the paradifferential flow \eqref{paradiff}, contained in the source term $f$. This estimate allow us to view the source perturbatively, which opens the way to using the paradifferential flow as a starting point for energy and Strichartz estimates.

\begin{theorem}[de Poyferre-Nguyen \cite{PN0}]\label{t:paralin}
Let 
\[
s > \frac32 + \frac{d}{2}, \quad 2 < r < s - \frac{d}{2} + \half.
\]
Suppose that $(\eta, \psi)$ is a solution of \eqref{zak} satisfying \eqref{width} and
\[
(\eta, \psi) \in L^\infty (I; H^{s + \half} (\R^d) \times H^s (\R^d)) \cap L^p (I; W^{r + \half, \infty}(\R^d) \times W^{r, \infty}(\R^d))
\]
with $p \in [1, \infty]$. Then
\[
u := T_p \eta + i T_q (\psi - T_B \eta)
\]
satisfies \eqref{paradiff} where
\begin{equation}\label{inhomog-est}
\|f(t) \|_{H^s} \leq \FF(\|\eta(t)\|_{H^{s + \half}}, \|\psi(t)\|_{H^s}) (1 + \|\eta(t)\|_{W^{r + \half,\infty}} + \|\psi(t)\|_{W^{r, \infty}}).
\end{equation}

\end{theorem}

To apply Corollary~\ref{c:var-strich-rough}, we record the regularity of the coefficients in \eqref{paradiff}. We have
\[
\|\D_\xi^{\beta}(\gamma \cdot P_\lambda) \|_{L_t^1([0, T]; L^\infty_\xi \dot C_x^{r - \half})} \leq \FF(\|\eta\|_{L_t^\infty([0, T]; L^\infty)}) \|\eta\|_{L_t^1([0, T]; C_x^{r + \half})} \cdot \lambda^{\frac32 - |\beta|}.
\]
Further, by for instance \cite[Lemma 3.8]{PN0},
\begin{equation}
\begin{aligned}
\|\D_\xi^{\beta}(V_{< \lambda} \cdot \xi \cdot P_\lambda) \|_{L_t^1([0, T]; L^\infty_\xi \dot C_x^{r - \half})} &\leq \FF(\|\eta\|_{L_t^\infty([0, T]; C_x^{r})}, \|\psi\|_{L_t^\infty([0, T]; C_x^{r - \half})})\\
&\quad \cdot (\|\eta\|_{L_t^1([0, T]; C_x^{r + \half})} + \|\psi\|_{L_t^1([0, T]; C_x^{r})})\cdot \lambda^{\frac32 - |\beta|}.
\end{aligned}
\end{equation}
Applying Corollary~\ref{c:var-strich-rough} with $r - \half$ in the place of $r$ and $m = \frac32$, we have
\[
\delta = \frac{2 - m}{2} = \frac14, \qquad \mu = \lambda^{\half} \lambda^{\frac{5 - 2r}{3 + 2r}}
\]
and
\begin{equation}
\begin{aligned}
\|u\|_{L^p([0, T];W^{s-\frac{8}{p(3 + 2r)}, q})} &\leq \FF(\|\eta\|_{L_t^\infty([0, T]; C_x^{r})}, \|\psi\|_{L_t^\infty([0, T]; C_x^{r - \half})}) \\
&\quad \cdot (\|\eta\|_{L_t^1([0, T]; C_x^{r + \half})} + \|\psi\|_{L_t^1([0, T]; C_x^{r})})\\
&\quad \cdot (\|u\|_{L^1([0, T];H^s)} + \|f\|_{L^1([0, T];H^{s - \half - \frac{5 - 2r}{3 + 2r}})}).
\end{aligned}
\end{equation}

Then by \eqref{inhomog-est}, a straightforward analysis of $u$ from definition, and Sobolev embedding, the right hand side may be rewritten as
\begin{equation}
\begin{aligned}
\|u\|_{L^p([0, T];W^{s-\frac{8}{p(3 + 2r)}, q})} &\leq \FF(\|\eta\|_{L_t^\infty([0, T]; H^{s + \half})}, \|\psi\|_{L_t^\infty([0, T]; H^s)}) \\
&\quad \cdot (\|\eta\|_{L_t^1([0, T]; C_x^{r + \half})} + \|\psi\|_{L_t^1([0, T]; C_x^{r})}).
\end{aligned}
\end{equation}
In the case $d = 1$, we choose $(p, q) = (4, \infty)$ and have 
\begin{equation}
\begin{aligned}
\|u\|_{L^4([0, T];W^{s-\frac{2}{2r + 3}, \infty})} &\leq \FF(\|\eta\|_{L_t^\infty([0, T]; H^{s + \half})}, \|\psi\|_{L_t^\infty([0, T]; H^s)}) \\
&\quad \cdot (\|\eta\|_{L_t^1([0, T]; C_x^{r + \half})} + \|\psi\|_{L_t^1([0, T]; C_x^{r})}).
\end{aligned}
\end{equation}
In the case $d \geq 2$, we choose $p = 2 + \eps$ and have by Sobolev embedding
\begin{equation}
\begin{aligned}
\|u\|_{L^{2 + \eps}([0, T];W^{s - \frac{d}{2} + \frac{2r - 1}{2r + 3} - \eps, \infty})} &\leq \FF(\|\eta\|_{L_t^\infty([0, T]; H^{s + \half})}, \|\psi\|_{L_t^\infty([0, T]; H^s)}) \\
&\quad \cdot (\|\eta\|_{L_t^1([0, T]; C_x^{r + \half})} + \|\psi\|_{L_t^1([0, T]; C_x^{r})}).
\end{aligned}
\end{equation}

Combined with a priori energy estimates of \cite{PN0}, we obtain the desired well-posedness result (see also the discussion at the end of \cite{Nsharp}).

\bibliography{bib-cap}
\bibliographystyle{plain}

\end{document}